\documentclass[10pt,a4paper,final,oneside,openright,reqno]{amsart} 

 \usepackage{latexsym, amsfonts, amsthm, amsmath, amssymb, tikz, enumitem, booktabs, float,color,multicol,epsf, url,epsfig,epstopdf, array, setspace, graphicx, csquotes, xcolor,tikz-cd}
\usepackage{amsfonts}
\usepackage{amsmath}
\usepackage{amssymb}
\usepackage[applemac]{inputenc}
\usepackage{url}
\usepackage{hyperref}
\usepackage{color}
\usepackage{lipsum}
\usepackage{mathtools}
 \usepackage[T1]{fontenc}
\usepackage{tikz-cd}
 \usepackage{url} 
 \usepackage{hyperref} 
\usepackage{pdfpages}
\usepackage{graphicx}

\newtheorem{theorem}{{Theorem}}[section]

\newtheorem{proposition}[theorem]{{Proposition}}

\newtheorem{lemma}[theorem]{{Lemma}}
\newtheorem{corollary}[theorem]{{Corollary}}

\newtheorem{remark}[theorem]{{Remark}}

\definecolor{greenbf}{rgb}{0, 0.7 ,0.3}
 \hypersetup{
colorlinks=true, breaklinks=true, urlcolor= blue, linkcolor= blue,
citecolor= {greenbf}, }

\def\C{\mathbb{C}}
\def\R{\mathbb{R}}

\def\H{\mathbb{H}}
\def\Z{\mathbb{Z}}
\def\S{\mathbb{S}}

\def\T{\mathbb{T}}

\def\Isom{{\sf{Isom}}}

\def\SOL{{\sf{SOL}}}
\def\Euc{{\sf{Euc}}}

\def\SOL{{\mathsf{Sol}} }

\def\sol{{\mathfrak{sol}} }
\def\euc{{\mathfrak{euc}} }
\def\gl{{\mathfrak{gl}} }

\def\Heis{{\sf{Heis}}}

\def\SO{{\sf{SO}}}
 
\def\SL{{\sf{SL}}}

\def\Span{{\sf{span}} }

\def\dS{{\sf{dS}} }

\def\det{{\sf{det}}}

\def\k{{\mathfrak{k}}}

\def\g{{\mathfrak{g}}}

\def\sl{{\mathfrak{sl}}}

\def\s{{\mathfrak{s}}}
\def\h{{\mathfrak{h}}}

\def\r{{\mathfrak{r}}}

 \newcommand{\ad}{\mathrm{ad}}
\newcommand{\Ad}{\mathrm{Ad}}
 
\begin{document}
\title[]{On closed geodesics in Lorentz manifolds}

\author{S. Allout} 
\address{Fakult\"at f\"ur Mathematik, Ruhr-Universit\"at Bochum, Germany}
\email{souheib.allout@rub.de}

\author{A. Belkacem}
\address{Department of mathematics, University of Batna 2, Algeria}
\email{abderrahmane.belkacem.matea@gmail.com}
 
\author{A. Zeghib}
\address{UMPA, CNRS, ENS de Lyon, France}
\email{abdelghani.zeghib@ens-lyon.fr}

\date{\today}

\begin{abstract}
 We construct compact Lorentz manifolds without closed geodesics. 
\end{abstract}

\maketitle

\section{Introduction} 
It is well known that (non-trivial) closed geodesics in compact Riemannian manifolds always exist (see for example \cite{Kli}). There has been, then, a lot of activity around the semi-Riemannian counterpart, mainly the Lorentzian case. Before going further, let us first introduce the following definition. A (non-trivial) geodesic $\gamma : I \to M$ in a semi-Riemannian manifold $(M, g)$ is called \textbf{\textit{weakly closed}} if there exists $s, t \in I$ with $t\neq s$, such that $\gamma(s) = \gamma(t)$ and $\dot{\gamma}(s)$ is proportional but not equal to $\dot{\gamma}(t)$ and called \textbf{\textit{closed}} if $\dot{\gamma}(s)=\dot{\gamma}(t)$. Observe that weakly closed geodesics are necessarily isotropic and incomplete. 

Our first result provides an example of a compact semi-Riemannian manifold without closed or weakly closed geodesics. More precisely, let $\SOL$ and $\Euc$ be the identity components of the isometry groups of the Minkowski and the Euclidean planes respectively and put $G=\SOL\times \Euc$. Then
\begin{theorem}\label{non periodic}
    There exists a left invariant metric on $G$, of signature $(2,4)$, and a cocompact lattice $\Gamma\subset G$ such that $\Gamma\backslash G$ is without closed or weakly closed geodesics.
\end{theorem}
The study of closed geodesics in semi-Riemannian manifolds received more attention in the Lorentzian case. The first result we want to mention in this direction is due to Tipler \cite{Tip} where it is shown that a compact spacetime with a covering space
containing a compact Cauchy surface admits closed timelike geodesics. The compactness assumption of the Cauchy surface is necessary as shown by Guediri \cite{Gue1} where he provides examples of compact flat Lorentz space forms without closed timelike geodesic, but they admit, however, closed lightlike geodesics.

In \cite{Gal2}, Galloway shows that every compact two-dimensional Lorentzian manifold contains at least one closed non-lightlike or (weakly) closed lightlike geodesic. He also constructed a three-dimensional example without closed non-spacelike geodesics. Galloway's result doesn't ensure existence of closed geodesics in Lorentzian surfaces. This was settled later by Suhr in \cite{Suh} where he shows that every compact Lorentzian surface contains at least two closed geodesics, one of them is non-lightlike and then constructs examples showing optimality of this bound. 

Galloway in \cite{Gal1}, Guediri in \cite{Gue1, Gue4}, S\'anchez in \cite{San} and Flores et al. in \cite{FJP} provide many existence (and non-existence) results of closed timelike geodesics in compact Lorentzian manifolds. See also \cite{CEG, Mas, Mas2} for further investigations.

The question whether a compact Lorentzian manifold admits closed geodesics remained open (see Question 9.1.1 in the recent survey \cite{BuM}). The following examples provide a negative answer:
\begin{theorem}\label{non closed}
    Let $G$ be either $\SL(2,\R)$ or $\SOL$. Then there exists a left invariant Lorentzian metric on $G$ such that every compact quotient $\Gamma\backslash G$ is without closed geodesics. However, $\Gamma\backslash G$ admits a countable collection of weakly closed null geodesics.
\end{theorem}

Unlike the locally homogeneous case, closed geodesics always exist in the homogeneous setting. More precisely
\begin{theorem}\label{homogeneous}
    A compact homogeneous semi-Riemannian manifold $(M,g)$ admits closed geodesics. If, in addition, $(M,g)$ is Lorentzian then it admits both timelike and spacelike closed geodesics.
\end{theorem}
\subsubsection*{Organization of the paper} Section \ref{proof homogeneous} is dedicated to the proof of Theorem \ref{homogeneous}. In Section \ref{LIM} we present some background about geodesics of left invariant metrics needed for the constructions later. We prove Theorems \ref{non closed} and \ref{non periodic} in Sections \ref{SolSL} and \ref{SolEuc} respectively.
\subsubsection*{Acknowledgements} We warmly thank the referee for the valuable remarks. We also thank Stefan Suhr for his helpful comments on this manuscript. The first author is fully supported by the SFB/TRR 191 `Symplectic Structures in Geometry, Algebra and Dynamics', funded by the DFG (Projektnummer 281071066 - TRR 191).
\section{The homogeneous case: proof of Theorem \ref{homogeneous}}\label{proof homogeneous}
The main purpose of this section is to show existence of closed geodesics in the case of compact homogeneous semi-Riemannian manifolds.
\begin{lemma}\label{precompact}
Let $X$ be a Killing vector field on a compact semi-Riemannian manifold $(M,g)$. The $X-$orbits of critical points of the function $p\mapsto g(X(p),X(p))$ are geodesics. If $X$ generates a precompact one-parameter subgroup of $\Isom(M,g)$ then it can be approximated by a Killing field $X^{\prime}$ whose flow defines a circle action. It follows that $(M,g)$ has closed geodesics. Moreover, $(M,g)$ has a spacelike (resp. timelike) closed geodesic if $X^{\prime}$ is spacelike (resp. timelike) somewhere.   
\end{lemma}
\begin{proof}
For the proof of the first statement see \cite{KN} Proposition 5.7 or \cite{FJP} Lemma 2.3 for more details. If $X$ generates a precompact one-parameter subgroup $I$ then its closure in $\Isom(M,g)$ is a compact connected torus $\T^k$ and it follows that $I$ can be approximated by closed one-parameter subgroups. The rest follows from the compactness of $M$.   
\end{proof}
\begin{corollary}\label{compact}
    A compact homogeneous semi-Riemannian manifold $(M,g)$, whose identity component of its isometry group is compact, admits both timelike and spacelike closed geodesics.
\end{corollary}

Now, suppose that $(M,g)$ is homogeneous and let $G=\Isom_0(M,g)$ be the identity component of the isometry group with Lie algebra $\g=\s \ltimes \r$ where $\s$ is the semisimple Levi factor and $\r$ is the solvable radical.

If $G$ admits a non-trivial compact semisimple Levi factor $K$, then applying Lemma \ref{precompact} one deduces existence of closed geodesics. Indeed, $K$ contains many closed one-parameter groups. In fact, the set of linear lines in the algebra $\k$ of $K$, generating closed one-parameter subgroups, is dense in the projective space $\mathbb{P}(\k)$.

On the other hand, if $\s$ has no compact factor then by \cite{BGZ} (Theorem A.) the $G-$action is locally free $i.e$ the isotropy subgroup $\Gamma\subset G$ is discrete. Moreover, the left invariant metric on $G$, obtained by pulling-back the metric $g$ to $G$, is in fact bi-invariant.
Therefore, geodesics in $G$ are right and left cosets of one-parameter subgroups. We distinguish two cases:

$\bullet$ \textbf{Case $\s\neq \{0\}$}: let $u\in \s$ be an elliptic element, that is, $\ad_u:\s\to \s$ is $\C-$diagonalizable with imaginary eigenvalues. Elliptic elements $u$ in $\s$ always exist and the $\ad_u-$action on $\g=\s\ltimes \r$ is also elliptic (in fact for any representation $\s\to \gl(V)$, the image of an elliptic element is elliptic). The right (or left) invariant Killing field $X_u$ determined by $u$ is equi-continuous $i.e$ it generates a precompact flow. Indeed, $\ad_u$ preserves a positive definite inner product on $\g$ and, hence, the left action of $\exp(tu)$ on $G/\Gamma$ preserves a right invariant Riemannian metric. Finally, by applying Lemma \ref{precompact} we deduce existence of closed geodesics.

$\bullet$ \textbf{Case $\s= \{0\}$}: so $\g=\r$ is solvable and, as discussed above, $M=G/\Gamma$ with $G$ endowed with a bi-invariant semi-Riemannian metric. Therefore, one-parameter groups in $G$ are geodesic and if $\gamma\in \Gamma$ belongs to a one-parameter group $I$, then $I$ projects to a closed geodesic in $G/\Gamma$. The exponential map for solvable groups fails to be surjective in general but it is, however, a diffeomorphism for nilpotent (simply connected) groups. Since every lattice in a solvable group intersects the nil-radical in a lattice, we deduce existence of closed geodesics in our case.
\subsection*{The Lorentzian case} If $(M,g)$ is a homogeneous compact Lorentzian manifold we can in fact deduce existence of both timelike and spacelike closed geodesics. Indeed, if the identity component $G$ of its isometry group is compact, then this follows from Corollary \ref{compact}. If $G$ is non-compact, then it follows by a classification in \cite{Zeg} that $(M,g)$ is covered by a metric product $H\times N$ where $N$ is a compact homogeneous Riemannian manifold and $H$ is a Lie group endowed with a bi-invariant metric. This Lie group is either $\widetilde{\SL}(2,\R)$ or an oscillator group $i.e$ an elliptic extension $\S^1\ltimes \Heis_{2n+1}$ of the Heisenberg group $\Heis_{2n+1}$ (let us mention, in fact, that general non necessarily transitive actions of Lie groups on compact Lorentz manifolds were also classified in \cite{AS1, AS2, Zeg2}). We deduce: 

$\bullet$ \textit{Closed spacelike geodesics}: if $N$ is non-trivial then we have closed spacelike geodesics. Suppose $N$ is trivial, so $M=H/\Gamma$ with $\Gamma\subset H$ a cocompact lattice. If $H=\widetilde{\SL}(2,\R)$ then any one-parameter subgroup intersecting $\Gamma$ in a hyperbolic element projects to a closed spacelike geodesic. If $H=\S^1\ltimes \Heis_{2n+1}$ then for the bi-invariant metric the subgroup $\Heis_{2n+1}$ is degenerate ($i.e$ lightlike) and totally geodesic. Since $\Gamma$ intersects $\Heis_{2n+1}$ in a lattice, we deduce existence of closed spacelike geodesics.

$\bullet$ \textit{Closed timelike geodesics}: In all cases, elliptic elements in the Lie algebra $\h$ of $H$ exist and are timelike with respect to the bi-invariant metric. They give rise to closed timelike geodesics.

\section{The geodesic equation for left invariant metrics}\label{LIM}
To a $C^1-$curve $\gamma:I\to G$ in a Lie group $G$, one associates the curve $D_{\gamma}:I\to \g$ in the Lie algebra $\g$ as follows: for $t\in I$ the velocity vector $\dot{\gamma}(t)$ lives in $T_{\gamma(t)}G$ which is identified with $\g$ via the left translation $L_{\gamma(t)}$. Put $D_{\gamma}(t)=(L_{\gamma(t)})_*^{-1}(\dot{\gamma}(t))$. One observes that $D_{\gamma}$ is constant if and only if $\gamma$ is the restriction to $I$ of a parameterized left coset of a one parameter group.\\

Suppose that $G$ is endowed with a left invariant semi-Riemannian metric, or equivalently, $\g$ is endowed with a semi-Riemannian inner product $\langle .,. \rangle$. Then a $C^2-$curve $\gamma:I\to G$ is a geodesic if and only if $D_{\gamma}$ solves the first order ODE, introduced in \cite{Arn}, which we refer to as the geodesic equation (also called the Euler-Arnold equation):  \begin{equation}\label{geo1}
    \dot{x}(t)=\ad_{x(t)}^{*}(x(t)).
\end{equation}

In other words, one has the vector field on $\g$ given by $x\mapsto \ad_{x}^{*}(x)$ and $\gamma$ is a geodesic if and only if $D_{\gamma}$ is a parameterized trajectory of the generated (local) flow. Clearly this vector field is $2$-homogeneous.

It is known that the geodesic equation (\ref{geo1}) admits at least one radial solution $\ad_v^*v=\lambda v$ with $v\neq 0$. Indeed, the fact that the geodesic vector field given by (\ref{geo1}) is $2$-homogeneous implies that it induces a map $\psi:\mathbb{P}^+(\g)\to \mathbb{P}^+(\g)$ if it does not vanish on $\g-\{0\}$, where $\mathbb{P}^+(\g)$ is the space of half lines from the origin. In this case, $\psi$ has even degree since it satisfies $\psi(x)=\psi(-x)$ which implies that it has fixed points. Also, observe that $\lambda\neq 0$ implies that $v$ is null.

Recall that $\langle .,. \rangle$ is bi-invariant if and only if $\ad_x^*=-\ad_x$ for all $x\in \g$, which implies that equation (\ref{geo1}) becomes $\dot{x}=0$. In this case all solutions are constant which means that geodesics in $G$ are nothing but left cosets of one-parameter groups.

Moreover, the case when the algebra $\g$ is quadratic $i.e$ $\g$ admits a bi-invariant semi-Riemannian inner product $\langle .,. \rangle$, the geodesic equation for any given metric on $\g$ can be simplified. More precisely, let $(\g, \langle .,. \rangle)$ be such an algebra and $q$ is any inner product on $\g$, then there is a unique $\langle .,. \rangle-$self-adjoint isomorphism $A_q:\g\to \g$ such that $q(v,w)=\langle v,A_q(w) \rangle$ for all $v,w\in \g$. The geodesic equation for $q$ can be rewritten as follows (see \cite{EFR} Proposition 4.2): 
\begin{equation}\label{geo2}
A_q(\dot{x})=[A_q(x),x]=-\ad_x(A_q(x)).
\end{equation}

\begin{remark}
    Let $K$ be a compact semisimple Lie group with Lie algebra $\k$ and Killing form $\kappa$. Let $A$ be a $\kappa-$self-adjoint isomorphism and $q(.,.)=\kappa(.,A.)$ the associated semi-Riemannian inner product. Consider the left invariant metric generated by $q$. By Lemma \ref{precompact} the left action of any one-parameter subgroup $I=\exp(tu)$ admits an orbit $Ig$ which is a geodesic. Thus, the one-parameter subgroup $g^{-1}Ig$ is a geodesic. In other words, it corresponds to a singular solution of the geodesic equation (\ref{geo2}) which means $[Au^{\prime},u^{\prime}]=0$ where $u^{\prime}=\Ad(g^{-1})u$. Therefore, every $u\in \k$ is conjugate to $u^{\prime}\in \k$ such that $u^{\prime}$ and $Au^{\prime}$ commute.
\end{remark}

\subsection{Dynamics of the geodesic flow} The fact that the vector field given by the ODE (\ref{geo1}) is $2$-homogeneous, implies that the scaling action sends solution to solution, up to affine reparameterization. This induces a (singular) foliation by curves on the projectivization $\mathbb{P}(\g)$ seen as the "dynamics" of the geodesic flow on the projective space. Similarly, we have a \textit{directed} foliation on the spherization $\pi:\g\to\mathbb{P}^+(\g)$. In the case of Lorentzian signature, this "flow" on $\mathbb{P}^+(\g)$ leaves invariant two conformal copies of the hyperbolic space $\H^{n-1}$, a conformal copy of the de Sitter space $\dS_{n-1}$, and two conformal Riemannian spheres $\pi(C^+), \pi(C^-)$ where $C^+$ and $C^-$ are the half null cones.

A solution in $\g$, of the geodesic equation (\ref{geo1}), is called \textit{direction-periodic} if it projects to a closed (possibly singular) trajectory in the spherization $\mathbb{P}^+(\g)$. Observe that both closed and weakly closed geodesics are direction-periodic.
\subsection{Compact quotients} Let $\Gamma\subset G$ be a cocompact lattice and consider the compact quotient $M = \Gamma \backslash G$. The left invariant metric on $G$ descends to $M$.  Projections to $M$ of left invariant vector fields are  fundamental vector fields of the right $G-$action on $M$. The tangent bundle $TG$ of $G$, under left translations, is identified with $G\times \g$. Therefore, we have a trivialisation $T M \to M \times \g$. For a curve $(x(t), v(t))$ in $T M$, one associates the curve $v(t)$ in $\g$, and this applies in particular to the case $v(t) = \dot{x}(t)$. A curve $x(t)$ in $M$ is a geodesic if and only if  $\dot{x}(t)$ solves the equation (\ref{geo1}) on $\g$.

Let $\Phi^t$ be the geodesic (local) flow. Then $\Phi^t(x, u) = (\phi(t, x, u), u(t))$, where $u(t)$ is a solution of the geodesic equation (\ref{geo1}). In particular, the map $(x, u) \in T M \mapsto u \in \g$ semi-conjugates the geodesic (local) flow on $T M$ to the (local) flow of the geodesic equation (\ref{geo1}). Let us pass to the projectivization of the tangent bundle $\mathbb{P}(T M ) = M \times \mathbb{P}(\g).$ The (singular) foliation on $\mathbb{P}(T M )$ induced by $\Phi^t$ projects to the (singular) foliation induced by the flow of the geodesic equation (\ref{geo1}) on $\mathbb{P}(\g)$.

Let $v\in \g$ be a constant solution of the equation ($\ref{geo1}$) $i.e$ $\ad_v^*v=0$. Let $I_v=\exp(tv)$ be the associated one-parameter group in $G$. The left cosets of $I_v$ are therefore geodesics. In other words, the orbits of the right action of $I_v$ on both $G$ and $M=\Gamma\backslash G$ are geodesics. Furthermore, a left coset $gI_v$ projects a closed geodesic in $M$ if and only if $gI_vg^{-1}\cap \Gamma$ is a lattice in $gI_vg^{-1}$.

On the other hand, if $v\in \g $ is a non-trivial radial solution, $i.e$ $\ad_v^*v=\lambda v$ with $\lambda\neq 0$, then the left cosets of the one-parameter group $I_v=\exp(tv)$ are geodesics but only up to parameterizations. Thus, the orbits of the right action are, up to reparameterizing, geodesics. Similarly, a left coset $gI_v$ projects a weakly closed geodesic, up to parameterizing, in $M$ if and only if $gI_vg^{-1}\cap \Gamma$ is a lattice in $gI_vg^{-1}$.

\section{No closed geodesics: proof of Theorem \ref{non closed}}\label{SolSL}
In this section we construct a left invariant Lorentzian metric on $G=\SL(2,\R)$ or $\SOL$ with the property that every compact quotient $\Gamma\backslash G$ admits no closed geodesics but they admit, however, a countable collection (up to reparameterizations) of weakly closed null geodesics.  For further details on the structure of Lie algebras and lattices in Lie groups we refer to \cite{Kir} and \cite{Rag}.

\subsection{$\SL(2,\R)$ case.} Endow the algebra $\sl(2,\R)$ with its Killing form $\langle .,.\rangle$ and let $e, h, f$ be a basis of $\sl(2,\R)$ such that $[f,e]=h \ \ [h,e]=-e \ \ [h,f]=f.$ Then $$\langle e,e\rangle=\langle e,h\rangle=\langle h,f\rangle=\langle f,f\rangle=0 \ \ \text{and} \ \ \langle e,f\rangle=\langle h,h\rangle=2.$$ 
Let $A$ be the $\langle .,.\rangle-$self-adjoint isomorphism of $\sl(2,\R)$ whose matrix with respect to $e, h, f$ is
$$A=\begin{pmatrix}
1 & 1 & 0\\
0 & 1 & 1\\
0 & 0 & 1
\end{pmatrix}.$$
Let $q$ be the Lorentzian metric given by $q(v,w)=\langle v, Aw\rangle$ for all $v,w\in \sl(2,\R)$. We have seen in (\ref{geo2}) that the geodesic equation for $q$ is $A\dot{v}=[Av,v]$.
\begin{lemma}\label{P}
    The plane $P=\Span(e,h)$ is invariant under the geodesic flow. More precisely, the geodesic vector field on $P$ has the form $v=(x,y)\in P\mapsto y^2e$. 
\end{lemma}
\begin{proof}
    Put $v=xe+yh$, then $[Av,v]=[(x+ y)e,yh]+[yh,xe]= [ ye,yh]=y^2e.$ Thus $A^{-1}[Av,v]=y^2e$.
\end{proof}
One observes, in fact, that the plane $P$ is an $A-$invariant subalgebra of $\sl(2,\R)$ isomorphic to the algebra of the affine group of the real line. This subalgebra, as Lemma \ref{P} shows, is lightlike and totally geodesic for the Lorentzian metric $q$.
\begin{corollary}
    All solutions with initial conditions in $P$ are complete with trajectories affine lines parallel to $\R e$.
\end{corollary}
One checks that the vector $v_0=(\frac{3}{8},\frac{-1}{2},1)$ with respect to the basis $e,h,f$ is null and satisfies $Av_0=[Av_0, v_0]$ and $\R v_0, \R e$ are the only radial directions. Therefore, the line $\R v_0$ is invariant under the geodesic flow and the geodesic vector field on $\R v_0$ is $\lambda v_0\mapsto \lambda^2 v_0$. Hence, every solution in $\R v_0-\{0\}$ is incomplete. More precisely, for $\lambda>0$ the solution through $\lambda v_0$ is defined on a maximal interval of the form $(-\infty,b)$ and for $\lambda<0$ it is defined on a maximal interval of the form $(a,+\infty)$. For more details about completeness of left invariant metrics on $\SL(2,\R)$ see \cite{BM} or \cite{EFR}.

Now, define the map $\sigma:\sl(2,\R)\to \R v_0$ to be the projection on $\R v_0$ with respect to the new basis $e, h, v_0$. We have

\begin{proposition}
    The projection $\sigma$ is equivariant with respect to the geodesic flow. 
\end{proposition}
\begin{proof}
    Let $v=w+cv_0$ with $w\in P=\ker(\sigma)$, then $$[Av,v]=[A(w+cv_0),w+cv_0]=[Aw,w]+c([Av_0,w]+[Aw,v_0])+[cAv_0,cv_0].$$
    Since $P$ is an $A-$invariant subalgebra we have $\sigma(A^{-1}[Av,v])=A^{-1}[A\sigma(v),\sigma(v)]$ if and only if $([Av_0,w]+[Aw,v_0])\in P$. Put $v_0=z+\alpha f$ for $z\in P$, then $$[Av_0,w]+[Aw,v_0]=[\alpha f+w_1,w]+[Aw,z+\alpha f]$$
    for $w_1\in P$. So it remains to show that $([f,w]+[Aw,f])\in P$. This follows from the fact that $Aw-w\in \R e$ and $[f,e]=h\in P$.
\end{proof}
\begin{corollary}
    A solution of the geodesic equation with initial condition $v\in \sl(2,\R)$ is complete if and only if $v\in P$, i.e, $\sigma(v)=0$. It is $\R_{+}$ (resp. $\R_{-}$) incomplete if $\sigma(v)>0$ (resp. $\sigma(v)<0$).
\end{corollary}
\begin{remark}\label{dyn} The dynamics on $\mathbb{P}^+(\sl(2,\R))$ has exactly four fixed points $e^+,e^-,v_0^+$ and $v_0^-$ corresponding to the half lines through $e$ and $v_0$. The plane $P$ corresponds to a circle $\mathbb{P}^+(P)$. Every point in $\mathbb{P}^+(P)$ different from $e^-$ and $e^+$ converges in the future to $e^+$ and to $e^-$ in the past. The three invariant circles $\mathbb{P}^+(P)$, $\pi(C^+)$, and $\pi(C^-)$ divide the sphere $\mathbb{P}^+(\sl(2,\R))$ into four invariant open disks, where $\pi(C^+)$ and $\pi(C^-)$ are the spherizations of the half null cones (see Figure (\ref{fig1})). Each of these open disks is invariant and solutions are unbounded inside it ($i.e$ they converge to the boundary of the disk). Indeed, a bounded solution forces the existence of a constant one ($i.e$ radial) inside the disk which is impossible. 
\end{remark}
\begin{figure}[h!]
    \centering
    \includegraphics[scale=0.09]{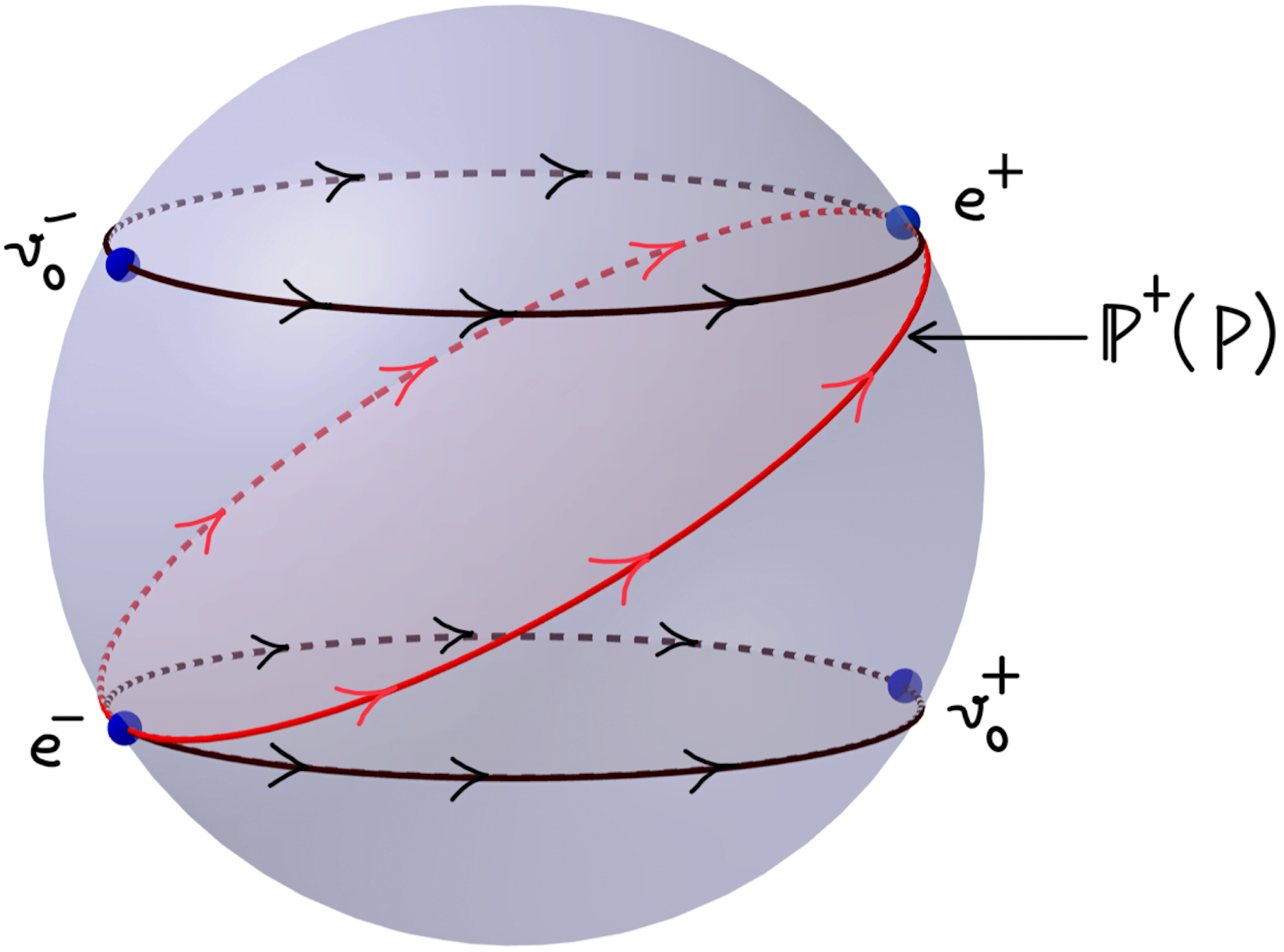}
    \caption{The dynamics on $\mathbb{P}^+(\sl(2,\R))$}
    \label{fig1}
\end{figure}

\begin{proposition} Let $g$ be the left invariant metric on $\SL(2,\R)$ generated by $q$ and $\Gamma\subset \SL(2,\R)$ be a cocompact lattice. Then $(\Gamma\backslash \SL(2,\R), g)$ is without closed geodesics. It admits, however, a countable collection of weakly closed geodesics. 
\end{proposition}
\begin{proof} Let $\delta$ be a closed or weakly closed geodesic in $(\Gamma\backslash \SL(2,\R), g)$ and $\widetilde{\delta}$ its maximal lift to $\SL(2,\R)$. Since the only closed orbits in $\mathbb{P}(\sl(2,\R))$ are the constant ones, then $\widetilde{\delta}$ is mapped to a constant solution $d_{\delta}\in \mathbb{P}(\sl(2,\R))$ which shows that $\widetilde{\delta}$ is a leaf of the left invariant line field generated by $d_{\delta}$. Hence, $\delta$ is, up to reparameterization, the projection of a left coset of the one-parameter group tangent to $d_{\delta}$. The line $d_{\delta}$ is either $\R e$ or $\R v_0$. But $\R e$ is nilpotent $i.e$ it generates a parabolic subgroup of $\SL(2,\R)$, hence the projection of each of its left cosets is dense in $\Gamma\backslash \SL(2,\R)$ (the right action of a parabolic subgroup on $\Gamma\backslash\SL(2,\R)$ is nothing but the horocyle flow). If $d_{\delta}=\R v_0$ then it is hyperbolic and $\delta$ is, up to parameterization, the projection of a left coset $gH$ of the hyperbolic one-parameter group $H$ generated by $\R v_0$. Moreover, a left coset $gH$ projects to a closed orbit if and only if $gHg^{-1}$ intersects $\Gamma$ non-trivially. Thus weakly closed null geodesics are in one-to-one correspondence with conjugates of $H$ intersecting $\Gamma$ non-trivially. This is a countable collection and all of them are incomplete. 
    
\end{proof}
\begin{remark}
    The left action of an elliptic one-parameter subgroup $K$ on $\SL(2,\R)$ is, by assumption, isometric. Let $X$ be its associated right invariant Killing vector field. The length function $p\in \SL(2,\R)\mapsto g(X(p),X(p))$ is without critical points. Indeed, the $X-$orbit of a critical point is a closed geodesic which projects to a closed geodesic in $\Gamma\backslash \SL(2,\R)$ and this contradicts the previous proposition. On the other hand, the vector field $Y$ on $\Gamma\backslash\SL(2,\R)$ that generates the right action of $K$ has constant length function, since it is the projection of a left invariant vector field, but clearly it is not Killing. The right $K-$action on $\Gamma\backslash \SL(2,\R)$ defines a fibration by (non-geodesic) circles which is locally homogeneous (the left action is defined locally on $\Gamma\backslash \SL(2,\R)$ and sends circle to circle). In particular, these circles have constant geodesic curvature and when $K$ converges to the parabolic one-parameter group generated by $e$, these fibrations converge to a foliation by dense null geodesics.
\end{remark}
\subsection{$\SOL$ case}\label{sol} The construction will be similar to the $\SL(2,\R)$ case. Let $\SOL$ be the solvable unimodular three-dimensional group $\R\ltimes \R^2$ where $\R$ acts on $\R^2$ via the representation $t\mapsto \begin{pmatrix}
e^t & 0\\
0 & e^{-t}
\end{pmatrix} $. This is the identity component of the isometry group of the quadratic form $(x,y)\mapsto xy$ on $\R^2$. Its Lie algebra $\sol$ has a basis $e_1, e_2, h$ with brackets $[h,e_1]=e_1 \ \ [h,e_2]=-e_2.$ Let $q$ be the Lorentzian inner product satisfying  $$q( e_1,e_1)=q( e_1,e_2)=q( h,e_2)=q( h,h)=0 \ \ \text{and} \ \ q( e_1,h)=q( e_2,e_2)=1$$ 
Therefore, with respect to this basis we have $$\ad_{e_1}=\begin{pmatrix}
0 & 0 & -1\\
0 & 0 & 0\\
0 & 0 & 0
\end{pmatrix} \ \ \ad_{e_2}=\begin{pmatrix}
0 & 0 & 0\\
0 & 0 & 1\\
0 & 0 & 0
\end{pmatrix} \ \ \ad_{h}=\begin{pmatrix}
1 & 0 & 0\\
0 & -1 & 0\\
0 & 0 & 0
\end{pmatrix}$$
and one checks that $$\ad_{e_1}^*=\begin{pmatrix}
0 & 0 & -1\\
0 & 0 & 0\\
0 & 0 & 0
\end{pmatrix} \ \ \ad_{e_2}^*=\begin{pmatrix}
0 & 1 & 0\\
0 & 0 & 0\\
0 & 0 & 0
\end{pmatrix} \ \ \ad_{h}^*=\begin{pmatrix}
0 & 0 & 0\\
0 & -1 & 0\\
0 & 0 & 1
\end{pmatrix}.$$

It is clear that the abelian subalgebra $P=\Span(e_1,e_2)$ is invariant under the geodesic flow and the geodesic vector field on $P$ has the form $v=(x,y,0)\in P\mapsto y^2e_1$ since for $v=xe_1 + ye_2$ we have $\ad_v^*v=(x\ad_{e_1}^*+y\ad_{e_2}^*)(xe_1+ye_2)=y^2e_1.$ One also checks easily that the equation $\ad_v^*v=\lambda v$ for some $\lambda\in \R$ admits exactly $\R e_1$ and $\R h$ as solutions  with $\ad_{e_1}^*e_1=0$ and $\ad_h^*h=h$. Therefore, the null line $\R h$ is invariant under the geodesic flow and every solution in $\R h-\{0\}$ is incomplete.

Similar to the $\SL(2,\R)$ define the projection $\sigma:\sol\to \R h$ with respect to the basis $e_1,e_2,h$. Then

\begin{proposition}
    The projection $\sigma$ is equivariant with respect to the geodesic flow. 
\end{proposition}
\begin{proof} Let $v=xe_1+ye_2+zh$, then $$\ad_v^*v=(x\ad_{e_1}^*+y\ad_{e_2}^*+z\ad_{h}^*)(xe_1+ye_2+zh)=(y^2-xz)e_1-yze_2+z^2h.$$
Therefore, $\sigma(\ad_v^*v)=z^2h=\ad_{zh}^*zh=\ad_{\sigma(v)}^*\sigma(v)$
    
\end{proof}

One observes that the situation is similar to the previous case of $\sl(2,\R)$, the dynamics on $\mathbb{P}^+(\sol)$ is as described in Remark \ref{dyn}. One concludes
\begin{corollary} Let $g$ be the left invariant metric on $\SOL$ generated by $q$ and $\Gamma\subset \SOL$ be a cocompact lattice. Then $(\Gamma\backslash \SOL , g)$ is without closed geodesics. It admits, however, a countable collection of weakly closed geodesics. 
    
\end{corollary}
\begin{proof}
    Let $\delta$ be a closed or weakly closed geodesic in $(\Gamma\backslash \SOL, g)$ and $\widetilde{\delta}$ its maximal lift to $\SOL$.  The geodesic $\widetilde{\delta}$ is mapped to a constant solution $d_{\delta}\in \mathbb{P}(\sol)$ for the same reason as in the $\sl(2,\R)$ case, which shows that $\widetilde{\delta}$ is a leaf of the left invariant line field generated by $d_{\delta}$. The line $d_{\delta}$ is either $\R e_1$ or $\R h$. But $\R e$ is impossible since $\Gamma$ intersects the stable and the unstable lines in $\R^2$ trivially. So $d_{\delta}=\R h$ and $\delta$ is, up to parameterization, the projection of a left coset $gH$ of the one-parameter group $H$ generated by $\R h$. Moreover, a left coset $gH$ projects to a closed orbit if and only if $gHg^{-1}$ intersects $\Gamma$ non-trivially. Thus weakly closed null geodesics are in one-to-one correspondence with conjugates of $H$ intersecting $\Gamma$ non-trivially. 
\end{proof}
\section{The $\SOL\times \Euc$ case: proof of Theorem \ref{non periodic}}\label{SolEuc} Let $\Euc=\SO(2)\ltimes \R^2$ be the identity component of the isometry group of the Euclidean plane. Its Lie algebra $\euc$ has a basis $f_1,f_2,e$ with brackets $[e,f_1]=-f_2$ and $[e,f_2]=f_1$. We have, with respect to this basis,
$$\ad_{f_1}=\begin{pmatrix}
0 & 0 & 0\\
0 & 0 & 1\\
0 & 0 & 0
\end{pmatrix} \ \ \ad_{f_2}=\begin{pmatrix}
0 & 0 & -1\\
0 & 0 & 0\\
0 & 0 & 0
\end{pmatrix} \ \ \ad_{e}=\begin{pmatrix}
0 & 1 & 0\\
-1 & 0 & 0\\
0 & 0 & 0
\end{pmatrix}.$$
Let $g$ be the Lorentz metric on $\euc$ satisfying $$g( f_1,f_1)=g( f_1,f_2)=g( e,f_2)=g( e,e)=0 \ \ \text{and} \ \ g( f_1,e)=g( f_2,f_2)=1.$$ 
One checks that $$\ad_{f_1}^*=\begin{pmatrix}
0 & 1 & 0\\
0 & 0 & 0\\
0 & 0 & 0
\end{pmatrix} \ \ \ad_{f_2}^*=\begin{pmatrix}
0 & 0 & -1\\
0 & 0 & 0\\
0 & 0 & 0
\end{pmatrix} \ \ \ad_{e}^*=\begin{pmatrix}
0 & 0 & 0\\
0 & 0 & 1\\
0 & -1 & 0
\end{pmatrix}.$$
The geodesic equation (\ref{geo1}) in this case is: for $v=(x,y,z)$ we have $$\ad_v^*v=\begin{pmatrix}
0 & x & -y\\
0 & 0 & z\\
0 & -z & 0
\end{pmatrix}\begin{pmatrix}
x \\
y \\
z 
\end{pmatrix}=\begin{pmatrix}
xy-yz \\
z^2 \\
-yz 
\end{pmatrix}.$$

Therefore, the abelian subalgebra $P=\{z=0\}$ is invariant under the geodesic flow and the geodesic vector field on $P$ has the form $(x,y,0)\in P\mapsto (xy,0,0)\in P$. Thus, on $P$, the constant solutions are the only periodic ones.

$\bullet$ \textit{Periodic solutions:} One sees that the vector field $(x,y,z)\mapsto (xy-yz, z^2,-yz)$ on the algebra $\euc$ is everywhere transverse to the plane distribution $\{y=0\}$ outside the subalgebra $P$. Hence, there are no periodic solutions in $\euc$ except the obvious constant ones inside the plane $P$.

$\bullet$ \textit{Radial solutions:} Suppose that $(xy-yz, z^2,-yz)=\lambda(x,y,z) $ for some $\lambda\neq 0$. Since there is no such a solution in $P$, then we can assume $z=1$. Thus, $(xy-y, 1,-y)=(\lambda x,\lambda y,\lambda) $ which implies $-y^2=1$ and this is impossible. Therefore, radial non-trivial solutions do not exist and constant solutions exist only in $P$.

$\bullet$ \textit{Direction-periodic solutions:}  The sphere $\mathbb{P}^+(\euc)$ is divided into four invariant open disks, bounded by $\mathbb{P}^+(P)$ and the two null circles similar to Figure (\ref{fig1}). A closed trajectory inside some open disk corresponds to a genuine periodic solution which is impossible as explained above. Therefore, direction-periodic  solutions are necessarily radial.
\subsection{The product $\SOL\times \Euc$}\label{product metric} Recall that in subsection \ref{sol} we constructed a Lorentz metric $q$ on $\sol$, endowed with the basis $e_1,e_2,h$, having the following properties: \begin{itemize}
    \item[$\circ$] Constant solutions of the geodesic equation are exactly the elements of $\R e_1$.
    \item[$\circ$] The line generated by $h$ is the only non-trivial radial direction.
    \item[$\circ$] There are no direction-periodic solutions except the radial ones.
\end{itemize}
Now, put the metric $q\oplus g$ on the algebra $\g=\sol\oplus \euc$ endowed with the product basis $e_1,e_2,h,f_1,f_2,e$. We have the immediate observations:
\begin{itemize}
    \item[$\circ$] Constant solutions for the geodesic equation on $\g$ project to constant solutions on both factors. Therefore, they consist of $\Span(e_1,f_1)\cup \Span(e_1,f_2)$.
    \item[$\circ$] Non-trivial radial solutions project to radial solutions with the same scaling factor. Thus, they consist of elements of $\R h-\{0\}$.
    \item[$\circ$] Direction-periodic solutions are radial and given by the above cases. Indeed, a direction-periodic solution projects to radial solutions on both factors, so it either corresponds to constant solutions on both factors or a non-trivial radial solution on the first factor and the zero solution on the second.
\end{itemize}

\subsection{A compact quotient without closed or weakly closed geodesics} Choose a hyperbolic element $A=\begin{pmatrix}
\lambda & 0 \\
0 & \lambda^{-1}
\end{pmatrix}$ and an irrational rotation $R_{\alpha}\in \SO(2)$ such that the $4\times 4$ matrix $\varphi=\begin{pmatrix}
A & 0 \\
0 & R_{\alpha}
\end{pmatrix}$ preserves a lattice $\Gamma_0\subset \R^4$, that is, $\varphi$ is conjugate to an element of $\SL(4,\Z)$. Such a map $\varphi$ exists (see Remark \ref{Borel}). Define the semi-direct product $\Gamma=\Z\ltimes \Gamma_0$ where $\Z$ acts on $\Gamma_0$ via $\varphi$. The discrete group $\Gamma$ is, in the obvious way, a cocompact lattice in $G=\SOL\times \Euc$. We have
\begin{proposition}
    Endow $G$ with the left invariant metric given by $q\oplus g$ on the algebra $\g=\sol\oplus \euc$. Then the compact quotient $\Gamma\backslash G$ admits no closed or weakly closed geodesics.
\end{proposition}
\begin{proof}
    Let $\delta$ be a (weakly) closed geodesic in $\Gamma\backslash G$ and $\widetilde{\delta}$ its maximal lift to $G$. Then $\widetilde{\delta}$ is mapped to a direction-periodic solution in $\g$. Discussion (\ref{product metric}) shows that $\widetilde{\delta}$ is mapped, in fact, to a radial direction $d_{\delta}$. Also, we have seen in (\ref{product metric}) that $d_{\delta}$ is either $\R h$ or any linear line contained in $\Span(e_1,f_1)\cup \Span(e_1,f_2)$. We claim that this is impossible in both cases. Indeed, $\widetilde{\delta}$ is, as an unparameterized curve, a left coset $gI_{\delta}$, of the one-parameter group $I_{\delta}$ tangent to $d_{\delta}$, such that $gI_{\delta}g^{-1}$ intersects $\Gamma$ non-trivially. But this is impossible because:
    \begin{enumerate}
        \item[$\star$] $\Gamma\cap \exp\bigl(\Span(e_1,f_1,f_2)\bigr)=\{0\}$ since the action of $\varphi$ on $\exp(\R e_1)$ is expanding and its action on $\exp(\Span(f_1,f_2))$ is an irrational rotation. Also, since $\exp\bigl(\Span(e_1,f_1,f_2)\bigr)$ is normal in $G$ then its intersection with $\Gamma$, even up to conjugacy, is trivial.
        \item[$\star$] All conjugates of $\exp(\R h)$ intersect $\Gamma$ trivially since the conjugacy action of every element of $\Gamma$ on $\R^4$ is either identically trivial or has an irrational rotational part.
    \end{enumerate}
    We conclude that closed or weakly closed geodesics do not exist in $\Gamma\backslash G$.
\end{proof}
\begin{remark}\label{Borel}
    By Borel-Harish-Chandra's theorem, the intersection $\Lambda=\SO(1,3)\cap \SL(4,\Z)$ is a lattice in $\SO(1,3)$. The lattice $\Lambda$ contains, in particular, many loxodromic elements. These are the elements of $\SO(1,3)$ that have exactly two fixed points when acting on the boundary at infinity of the hyperbolic space $\H^3$. They admit, then,  hyperbolic and elliptic factors. If the elliptic part of a loxodromic element in $\Lambda$ has finite order then its invariant plane $P$ is rational $i.e$ $P\cap \Z^4$ is a lattice in $P$. Since,  up to conjugacy, there are only finitely many elliptic elements in $\SL(2,\Z)$, then there is $k$ such that $A^k$ fixes a plane for every loxodromic element $A\in \Lambda$ whose elliptic part is of finite order. Suppose $\Lambda$ is torsion free (this is always possible up to finite index). If all loxodromic elements of $\Lambda$ have finite order elliptic parts then the (polynomial) function $f:A\in \Lambda\mapsto \det(A^k-Id)$ vanishes on $\Lambda$ (since $\Lambda$ in this case contains only unipotent and loxodromic elements). But, Borel's density theorem implies that $f$ vanishes identically on $\SO(1,3)$ which is impossible. Thus, $\Lambda$ contains (in fact many) loxodromic elements with irrational elliptic parts.
\end{remark}

\end{document}